\documentclass[11pt]{article}
\usepackage{amsfonts, amsmath, amsthm, amssymb, latexsym, amscd, color}
\usepackage{tabularx,ragged2e,booktabs,caption}
\usepackage{rotating}
\usepackage{mathrsfs,calrsfs}
\usepackage{textcomp}
\usepackage{graphicx,tikz}
\newcommand{\comment}[1]{}
\newif\ifpdf
\ifpdf \else \fi \textwidth = 6.5 in \textheight = 9 in
\newtheorem{thm}{Theorem}[section]

\newtheorem{observation}[thm]{Observation}

\newtheorem{corollary}[thm]{Corollary}
\newtheorem{lemma}[thm]{Lemma}
\newtheorem{definition}[thm]{Definition}
\newtheorem{theorem}[thm]{Theorem}

\begin{document}

\title{Total domination in cubic Kn\"odel graphs}
\author{{\small D.A. Mojdeh$^{a}$, S.R. Musawi}$^{b}$, {\small E. Nazari$^{b}$} and {\small N. Jafari Rad$^{c}$}\\
\\{\small $^{a}$Department of
Mathematics, University of Mazandaran,}\\ {\small Babolsar, Iran}\\
{\small Email: damojdeh@yahoo.com}\\
{\small $^{b}$Department of Mathematics, University of Tafresh,}\\ {\small Tafresh, Iran}\\
{\small $^{c}$Department of Mathematics, Shahrood University of Technology,}\\ {\small Shahrood, Iran}\\
{\small Email: n.jafarirad@gmail.com}}
\date{}
\maketitle

\begin{abstract}
A subset $D$ of vertices of a graph $G$ is a \textit{dominating
set} if for each $u\in V(G) \setminus D$, $u$ is adjacent to some
vertex $v\in D$. The \textit{domination number}, $\gamma(G)$ of
$G$, is the minimum cardinality of a dominating set of $G$. A set
$D\subseteq V(G)$ is a \textit{total dominating set} if for each
$u\in V(G)$, $u$ is adjacent to some vertex $v\in D$. The
\textit{total domination number}, $\gamma_{t}(G)$ of $G$, is the
minimum cardinality of a total dominating set of $G$. For an even
integer $n\ge 2$ and $1\le \Delta \le \lfloor\log_2n\rfloor$, a
\textit{Kn\"odel graph} $W_{\Delta,n}$ is a $\Delta$-regular
bipartite graph of even order $n$, with vertices $(i,j)$, for
$i=1,2$ and $0\le j\le n/2-1$, where for every $j$, $0\le j\le
n/2-1$, there is an edge between vertex $(1, j)$ and every vertex
$(2,(j+2^k-1)$ mod (n/2)), for $k=0,1,\cdots,\Delta-1$. In this
paper, we determine the total domination number in $3$-regular
Kn\"odel graphs $W_{3,n}$.

\end{abstract}

\section{introduction}
For graph theory notation and terminology not given here, we refer
to \cite{hhs}. Let $G=(V,E)$ denote a simple graph of order
$n=|V(G)|$ and size $m=|E(G)|$. Two vertices $u,v\in V(G)$ are
\textit{adjacent} if $uv\in E(G)$. The \textit{open neighborhood}
of a vertex $u\in V(G)$ is denoted by $N(u)=\{v\in V(G) | uv\in
E(G)\}$ and for a vertex set $S\subseteq V(G)$,
$N(S)=\underset{u\in S}{\cup}N(u)$. The cardinality of $N(u)$ is
called the \textit{degree} of $u$ and is denoted by $\deg(u)$,
(or $\deg_G(u)$ to refer it to $G$). The \textit{maximum degree}
and \textit{minimum degree} among all vertices in $G$ are denoted
by $\Delta(G)$ and $\delta(G)$, respectively. A graph $G$ is a
\textit{bipartite graph} if its vertex set can partition to two
disjoint sets $X$ and $Y$ such that each edge in $E(G)$ connects
a vertex in $X$ with a vertex in $Y$. A set $D\subseteq V(G)$ is
a \textit{dominating set} if for each $u\in V(G) \setminus D$,
$u$ is adjacent to some vertex $v\in D$. The \textit{domination
number}, $\gamma(G)$ of $G$, is the minimum cardinality of a
dominating set of $G$. A set $D\subseteq V(G)$ is a\textit{ total
dominating set} if for each $u\in V(G)$, $u$ is adjacent to some
vertex $v\in D$. The \textit{total domination number},
$\gamma_{t}(G)$ of $G$, is the minimum cardinality of a total
dominating set of $G$. The concept of domination theory is a
widely studied concept in graph theory and for a comprehensive
study see, for example \cite{hhs,hy}.

An interesting family of graphs namely \textit{Kn\"odel graphs}
have been introduced about 1975 \cite{k}, and have been studied
seriously by some authors since 2001, see for example
\cite{bhlp,fr0,fr00,fr}. For an even integer $n\ge 2$ and $1\le
\Delta \le \lfloor\log_2n\rfloor$, a \textit{Kn\"odel graph}
$W_{\Delta,n}$ is a $\Delta$-regular bipartite graph of even
order $n$, with vertices $(i,j)$, for $i=1,2$ and $0\le j\le
n/2-1$, where for every $j$, $0\le j\le n/2-1$, there is an edge
between vertex $(1, j)$ and every vertex $(2,(j+2^k-1)$ mod
(n/2)), for $k=0,1,\cdots,\Delta-1$ (see \cite{xxyf}). Kn\"odel
graphs, $W_{\Delta,n}$, are one of the three important families
of graphs that they have good properties in terms of broadcasting
and gossiping, see for example \cite{gh}. It is worth-noting that
any Kn\"odel graph is a Cayley graph and so it is a
vertex-transitive graph (see \cite{fr}).

Xueliang et. al. \cite{xxyf} studied the domination number in
$3$-regular Kn\"odel graphs $W_{3,n}$. They obtained exact
domination number for $W_{3,n}$. In this paper, we determine the
total domination number in $3$-regular Kn\"odel graphs $W_{3,n}$.
In Section 2, we prove some properties in the Kn\"odel graphs. In
Section 3, we present the total domination number in the
$3$-regular Kn\"odel graphs $W_{3,n}$. We need the following
simple observation from number theory.

\begin{observation}\label{obs0}
If $a$, $b$, $c$, $d$ and $x$ are positive integers such that
$x^a-x^b=x^c-x^d\ne0$, then $a=c$ and $b=d$.
\end{observation}

\section{Properties in the Kn\"odel graphs}

For simplicity, in this paper, we re-label the vertices of a
Kn\"odel graph as follows: we label $(1,i)$ by $u_{i+1}$ for each
$i=0,1,...,n/2-1$, and $(2,j)$ by $v_{j+1}$ for
$j=0,1,...,n/2-1$. Let $U=\{u_1,u_2,\cdots, u_{\frac{n}{2}}\}$
and $V=\{v_1,v_2,\cdots, v_{\frac{n}{2}}\}$. From now on, the
vertex set of each Kn\"odel graph $W_{\Delta,n}$ is $U\cup V$
such that $U$ and $V$ are the two partite sets of the graph. If
$S$ is a set of vertices of $W_{\Delta,n}$, then clearly, $S\cap
U$ and $S\cap V$ partition $S$, $|S|=|S\cap U|+|S\cap V|$,
$N(S\cap U)\subseteq V$ and $N(S\cap V)\subseteq U$. Note that
two vertices $u_i$ and $v_j$ are adjacent if and only if $j\in
\{i+2^0-1,i+2^1-1,\cdots,i+2^{\Delta-1}-1\}$, where the addition
is taken in modulo $n/2$. Figure 1, shows new labeling of
Kn\"odel graphs $W_{3,8}, W_{3,10}$ and $W_{3,12}$.

\begin{figure}
\begin{center}
\begin{tikzpicture}[scale=1,thick]
\draw (0,0)--(0,2);\draw (0,0)--(1,2);\draw (0,0)--(3,2); \draw
(1,0)--(1,2);\draw (1,0)--(2,2);\draw (1,0)--(0,2); \draw
(3,0)--(3,2);\draw (3,0)--(0,2);\draw (3,0)--(2,2); \draw
(2,0)--(2,2);\draw (2,0)--(3,2);\draw (2,0)--(1,2); \node at
(0,2.2) {$v_{1}$};\node at (1,2.2) {$v_{2}$};\node at (2,2.2)
{$v_{3}$};\node at (3,2.2) {$v_{4}$}; \node at (0,-0.2)
{$u_{1}$};\node at (1,-0.2) {$u_{2}$};\node at (2,-0.2)
{$u_{3}$};\node at (3,-0.2) {$u_{4}$}; \node at (1.5,-0.6)
{$W_{3,8}$};

\draw (4.5,0)--(4.5,2);\draw (4.5,0)--(5.5,2);\draw
(4.5,0)--(7.5,2); \draw (5.5,0)--(5.5,2);\draw
(5.5,0)--(6.5,2);\draw (5.5,0)--(8.5,2); \draw
(6.5,0)--(6.5,2);\draw (6.5,0)--(7.5,2);\draw (6.5,0)--(4.5,2);
\draw (7.5,0)--(7.5,2);\draw (7.5,0)--(8.5,2);\draw
(7.5,0)--(5.5,2); \draw (8.5,0)--(8.5,2);\draw
(8.5,0)--(4.5,2);\draw (8.5,0)--(6.5,2); \node at (4.5,2.2)
{$v_{1}$};\node at (5.5,2.2) {$v_{2}$};\node at (6.5,2.2)
{$v_{3}$};\node at (7.5,2.2) {$v_{4}$};\node at (8.5,2.2)
{$v_{5}$}; \node at (4.5,-0.2) {$u_{1}$};\node at (5.5,-0.2)
{$u_{2}$};\node at (6.5,-0.2) {$u_{3}$};\node at (7.5,-0.2)
{$u_{4}$};\node at (8.5,-0.2) {$u_{5}$}; \node at (6.5,-0.6)
{$W_{3,10}$};

\filldraw[black] (0,0) circle (1.5pt);
\draw (10,0)--(10,2);\draw (10,0)--(11,2);\draw (10,0)--(13,2);
\draw (11,0)--(11,2);\draw (11,0)--(12,2);\draw (11,0)--(14,2);
\draw (12,0)--(12,2);\draw (12,0)--(13,2);\draw (12,0)--(15,2);
\draw (13,0)--(13,2);\draw (13,0)--(14,2);\draw (13,0)--(10,2);
\draw (14,0)--(14,2);\draw (14,0)--(15,2);\draw (14,0)--(11,2);
\draw (15,0)--(15,2);\draw (15,0)--(10,2);\draw (15,0)--(12,2);
\node at (10,2.2) {$v_{1}$};\node at (11,2.2) {$v_{2}$};\node at
(12,2.2) {$v_{3}$};\node at (13,2.2) {$v_{4}$};\node at (14,2.2)
{$v_{5}$};\node at (15,2.2) {$v_{6}$}; \node at (10,-0.2)
{$u_{1}$};\node at (11,-0.2) {$u_{2}$};\node at (12,-0.2)
{$u_{3}$};\node at (13,-0.2) {$u_{4}$};\node at (14,-0.2)
{$u_{5}$};\node at (15,-0.2) {$u_{6}$}; \node at (12.5,-0.6)
{$W_{3,12}$};

\filldraw[black] (0,0) circle (1.5pt);\filldraw[black] (1,0) circle (1.5pt);\filldraw[black] (2,0) circle (1.5pt);\filldraw[black] (3,0) circle (1.5pt);
\filldraw[black] (0,2) circle (1.5pt);\filldraw[black] (1,2) circle (1.5pt);\filldraw[black] (2,2) circle (1.5pt);\filldraw[black] (3,2) circle (1.5pt);
\filldraw[black] (4.5,0) circle (1.5pt);\filldraw[black] (5.5,0) circle (1.5pt);\filldraw[black] (6.5,0) circle (1.5pt);\filldraw[black] (7.5,0) circle (1.5pt);\filldraw[black] (8.5,0) circle (1.5pt);
\filldraw[black] (4.5,2) circle (1.5pt);\filldraw[black] (5.5,2) circle (1.5pt);\filldraw[black] (6.5,2) circle (1.5pt);\filldraw[black] (7.5,2) circle (1.5pt);\filldraw[black] (8.5,2) circle (1.5pt);
\filldraw[black] (10,0) circle (1.5pt);\filldraw[black] (11,0) circle (1.5pt);\filldraw[black] (12,0) circle (1.5pt);\filldraw[black] (13,0) circle (1.5pt);\filldraw[black] (14,0) circle (1.5pt);\filldraw[black] (15,0) circle (1.5pt);
\filldraw[black] (10,2) circle (1.5pt);\filldraw[black] (11,2) circle (1.5pt);\filldraw[black] (12,2) circle (1.5pt);\filldraw[black] (13,2) circle (1.5pt);\filldraw[black] (14,2) circle (1.5pt);\filldraw[black] (15,2) circle (1.5pt);
\end{tikzpicture}
\end{center}
\caption{New labeling of Kn\"odel graphs $W_{3,8}, W_{3,10}$ and
$W_{3,12}$.} \label{fig:M1}
\end{figure}
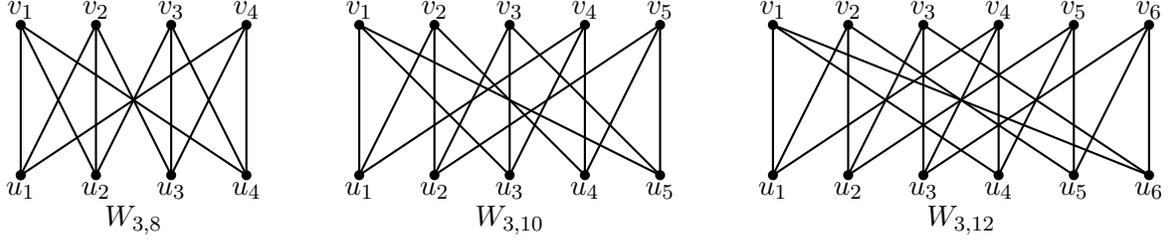
For any subset $\{u_{i_1},u_{i_2},\cdots,u_{i_k}\}$ of $U$ with
$1\le i_1 <i_2<\cdots< i_k\le\frac{n}{2}$, we correspond a
sequence based on the differences of the indices of $u_j$,
$j=i_1,...,i_k$, as follows.

\begin{definition}\label{def1}
For any subset $A=\{u_{i_1},u_{i_2},\cdots,u_{i_k}\}$ of $U$ with
$1\le i_1 <i_2<\cdots< i_k\le\frac{n}{2}$ we define a sequence
$n_1, n_2, \cdots, n_k$, namely \textbf{cyclic-sequence}, where
$n_j=i_{j+1}-i_j$ for $1\le j\le k-1$ and
$n_k=\frac{n}{2}+i_1-i_k$. For two vertices $u_{i_j},
u_{i_{j'}}\in A$ we define \textbf{index-distance} of  $u_{i_j}$
and $u_{i_{j'}}$ by $id(u_{i_j},
u_{i_{j'}})=min\{|i_j-i_{j'}|,\frac{n}{2}-|i_j-i_{j'}|\}$.
\end{definition}

\begin{observation}\label{obs1}
Let $A=\{u_{i_1},u_{i_2},\cdots,u_{i_k}\}\subseteq U$ be a set
such that $1\le i_1 <i_2<\cdots< i_k\le\frac{n}{2}$ and let $n_1,
n_2, \cdots, n_k$ be the corresponding cyclic-sequence of $A$.
Then,\\
(1) $n_1+n_2+\cdots+n_k=\frac{n}{2}$.\\
(2) If $u_{i_j}, u_{i_{j'}}\in A$, then $id(u_{i_j}, u_{i_{j'}})$
equals to sum of some consecutive elements of the cyclic-sequence
of $A$ and $\frac{n}{2}-id(u_{i_j}, u_{i_{j'}})$ is sum of the
remaining elements of the cyclic-sequence. Furthermore,
$\{id(u_{i_j},u_{i_{j'}}),\frac{n}{2}-id(u_{i_j},
u_{i_{j'}})\}=\{|i_j-i_{j'}|,\frac{n}{2}-|i_j-i_{j'}|\}$.
\end{observation}

We henceforth use the notation $\mathscr{M}_{\Delta}=\{ 2^a-2^b:0 \leq b<a< \Delta \}$ for $\Delta\geq 2$.

\begin{lemma}\label{lem0}
In the Kn\"odel graph $W_{\Delta,n}$ with vertex set $U\cup V$, for two distinct vertices $u_i$ and $u_j$,
$N(u_i)\cap N(u_j)\ne \emptyset$ if and only if $id(u_i,u_j)\in
\mathscr{M}_{\Delta}$ or $\frac{n}{2}-id(u_i,u_j)\in
\mathscr{M}_{\Delta}$.
\end{lemma}
\begin{proof}
Since $W_{\Delta,n}$ is vertex-transitive, for simplicity, we put
$1=i<j\le\frac{n}{2}$. We have $id(u_1,u_j)=\min\{j-1,
\frac{n}{2}-(j-1)\}$ and so $\frac{n}{2}-id(u_1,u_j)\}=\max\{j-1,
\frac{n}{2}-(j-1)\}$. Also, we have $N(u_1)=\{v_1, v_2,v_4,
\cdots, v_{2^{\Delta-1}}\}$ and $N(u_j)=\{v_j, v_{j+1},v_{j+3},
\cdots, v_{j+2^{\Delta-1}-1}\}$. First assume that $N(u_1)\cap
N(u_j)\ne \emptyset$. Let $v_k\in N(u_1)\cap N(u_j)$. There exist
two integers $a$ and $b$ such that $0\le a,b\le \Delta-1$ and
$k\equiv2^a\equiv j+2^b-1$(mod n/2). Since $1\le 2^a, 2^b, j\le
\frac{n}{2}$, we have $1\le j+2^b-1<n$. If $1\le j+2^b-1\le
\frac{n}{2}$, then $2^a=j+2^b-1$ and
$j-1=2^a-2^b\in\mathscr{M}_{\Delta}$ and if $\frac{n}{2}<
j+2^b-1<n$, then $2^a=j+2^b-1-\frac{n}{2}$ and
$\frac{n}{2}-(j-1)=2^b-2^a\in\mathscr{M}_{\Delta}$. Therefore, by
Observation \ref{obs1}, $id(u_i,u_j)\in \mathscr{M}_{\Delta}$ or
$\frac{n}{2}-id(u_i,u_j)\in \mathscr{M}_{\Delta}$.

Conversely, suppose $id(u_1,u_j)\in\mathscr{M}_{\Delta}$ or
$\frac{n}{2}-id(u_1,u_j)\in\mathscr{M}_{\Delta}$. Then
$j-1\in\mathscr{M}_{\Delta}$ or
$\frac{n}{2}-(j-1)\in\mathscr{M}_{\Delta}$. If
$j-1\in\mathscr{M}_{\Delta}$, then we have $j-1=2^a-2^b$ for two
integers $0\le a,b\le\Delta-1$. Then $2^a=j+2^b-1$ and $v_{2^a}\in
N(u_1)\cap N(u_j) $. If
$\frac{n}{2}-(j-1)\in\mathscr{M}_{\Delta}$, then we have
$\frac{n}{2}-(j-1)=2^c-2^d$ for two integers $0\le
c,d\le\Delta-1$. Now $2^c=j+2^d-1-\frac{n}{2}\equiv j+2^d-1$(mod
n/2) and $v_{2^c}\in N(u_1)\cap N(u_j) $. Thus in each case,
$N(u_i)\cap N(u_j)\ne\emptyset$.
\end{proof}

\begin{lemma}\label{lem00}
In the Kn\"odel graph $W_{\Delta,n}$ with vertex set $U\cup V$, for two distinct vertices $u_i$ and $u_j$,
$|N(u_i)\cap N(u_j)|=2$ if and only if $id(u_i,u_j)\in
\mathscr{M}_{\Delta}$ and $\frac{n}{2}-id(u_i,u_j)\in
\mathscr{M}_{\Delta}$.
\end{lemma}

\begin{proof} Without loss of generality, we assume that $1\le i<j\le \frac{n}{2}$.
Suppose that $|N(u_i)\cap N(u_j)|=2$ and $v_k,v_{k'}\in
N(u_i)\cap N(u_j)$ are two distinct vertices in $V$. There exist
two integers $a$ and $b$ such that $0\le a,b\le \Delta-1$ and
$k\equiv i+2^a-1\equiv j+2^b-1$(mod n/2). Similarly, there exist
two integers $a'$ and $b'$ such that $0\le a',b'\le \Delta-1$ and
$k'\equiv i+2^{a'}-1\equiv j+2^{b'}-1$(mod n/2). Now we have
$j-i\equiv 2^b-2^a\equiv 2^{b'}-2^{a'}$(mod n/2). We know that
$-\frac{n}{2}< 2^b-2^a, 2^{b'}-2^{a'} <\frac{n}{2}$. If
$-\frac{n}{2}< 2^b-2^a, 2^{b'}-2^{a'} <0$ or $0<2^b-2^a,
2^{b'}-2^{a'} <\frac{n}{2}$, then we have
$2^b-2^a=2^{b'}-2^{a'}\ne0$. Observation \ref{obs0} implies that $b=b'$ and
therefore $k\equiv k'$(mod n/2) and  $v_k=v_{k'}$, a
contradiction. By symmetry, we assume that
$0<2^b-2^a<\frac{n}{2}$ and $-\frac{n}{2}<2^{b'}-2^{a'} <0$. Since
$0<j-i<\frac{n}{2}$, we have $j-i=2^b-2^a$ and
$\frac{n}{2}-(j-i)=2^{a'}-2^{b'}$ which implies that $j-i \in
\mathscr{M}_{\Delta}$ and $\frac{n}{2}-(j-i) \in
\mathscr{M}_{\Delta}$. Thus by Observation \ref{obs1},
$id(u_i,u_j)\in \mathscr{M}_{\Delta}$ and
$\frac{n}{2}-id(u_i,u_j)\in \mathscr{M}_{\Delta}$.

Conversely, assume that $id(u_i,u_j)\in \mathscr{M}_{\Delta}$ and
$\frac{n}{2}-id(u_i,u_j)\in \mathscr{M}_{\Delta}$ for two distinct vertices $u_i$ and $u_j$. There exist
two integers $a$ and $b$ such that $0\le b<a\le \frac{n}{2}$ and
$j-i=2^a-2^b$. Also there exist two integer $a'$ and $b'$ such
that $0\le a'<b'\le \frac{n}{2}$ and
$\frac{n}{2}-(j-i)=2^{b'}-2^{a'}$. Now we have $i+2^a-1=j+2^b-1$
and $i+2^{a'}-1=j+2^{b'}-1-\frac{n}{2}\equiv j+2^{b'}-1$(mod
n/2). We set $k=i+2^a-1$ and $k'=i+2^{a'}-1$. Then $v_k,v_{k'}\in
N(u_i)\cap N(u_j)$ and $|N(u_i)\cap N(u_j)|\ge2$. Notice that
$k\not\equiv k'$(mod n/2), since otherwise $a=a'$ and
$2^{b'}-2^b=\frac{n}{2}$, a contradiction. Suppose that
$|N(u_i)\cap N(u_j)|\geq3$. Let $v_k,v_{k'},v_{k''}\in N(u_i)\cap
N(u_j)$ be three distinct vertices. Similar to the first part of the proof, for $v_k$ and $v_{k'}$,
there exist two
integers $a''$ and $b''$ such that $0\le a'',b''\le \Delta-1$ and
$k''\equiv i+2^{a''}-1\equiv j+2^{b''}-1$(mod n/2) and thus
$j-i\equiv 2^{a''}-2^{b''}$(mod n/2). Since $u_i$ and $u_j$ are disctinct, we have $a''\neq b''$. If $a''>b''$, then
$j-i=2^{a''}-2^{b''}$ and it can be seen that $j-i=\frac{n}{2}-(2^a
-2^b)=\frac{n}{2}-(2^{a'}-2^{b'})$ and Observation \ref{obs0} implies that $a=a'$ and thus $v_k=v_{k'}$, a
contradiction. If $a''<b''$, then
$j-i=\frac{n}{2}-(2^{a''}-2^{b''})$ and it can be seen that
$j-i=2^a-2^b=2^{a'}-2^{b'}$ and Observation \ref{obs0} implies that $a=a'$, a contradiction. Consequently
$|N(u_i)\cap N(u_j)|=2$.
\end{proof}

\begin{corollary}\label{cor1} (i) In the Kn\"odel graph $W_{\Delta,n}$ with vertex set $U\cup V$, for each $1\le i<j\le n/2$, $|N(u_i)\cap N(u_j)|=1$ if and only if precisely one of the values $id(u_i,u_j)$ and $\frac{n}{2}-id(u_i,u_j)$ belongs to $\mathscr{M}_{\Delta}$.\\
(ii) In the Kn\"odel graph $W_{\Delta,n}$, there exist distinct
vertices with two common
neighbors if and only if $n=2^a-2^b+2^c-2^d$ and
$a>b\ge1,c>d\ge1$.
\end{corollary}

\begin{corollary}\label{cor2}
Any three vertices in the Kn\"odel graph $W_{\Delta,n}$ have at
most one common neighbor. Indeed, any Kn\"odel graph is a
$K_{2,3}$-free graph.
\end{corollary}

\begin{lemma}\label{lem000}
In the Kn\"odel graph $W_{\Delta,n}$ with vertex set $U\cup V$ and $\Delta<\emph{log}_2(n/2+2)$, we have:\\
(i) $|N(u_i)\cap N(u_j)|\le1$, $1\le i<j\le n/2$.\\
(ii) $|N(u_i)\cap N(u_j)|=1$ if and only if $id(u_i,u_j)\in \mathscr{M}_{\Delta}$.
\end{lemma}
\begin{proof}
(i) Suppose to the contrary that $|N(u_i)\cap N(u_j)|>1$, then by
Corollary \ref{cor2} we have $|N(u_i)\cap N(u_j)|=2$. Then the
Lemma \ref{lem00} implies that $id(u_i,u_j)\in
\mathscr{M}_{\Delta}$ and $\frac{n}{2}-id(u_i,u_j)\in
\mathscr{M}_{\Delta}$. Thus $id(u_i,u_j)\le 2^{\Delta-1}-1$,
$\frac{n}{2}-id(u_i,u_j)\le 2^{\Delta-1}-1$ and $\frac{n}{2}\le
2^{\Delta}-2$. This inequality implies that
$\Delta\ge\emph{log}_2(n/2+2)$, a contradiction. Hence
$|N(u_i)\cap N(u_j)|\le1$, as desired.

(ii) Assume that $|N(u_i)\cap N(u_j)|=1$. By Corollary \ref{cor1},
precisely one of the values $id(u_i,u_j)$ and
$\frac{n}{2}-id(u_i,u_j)$ belongs to $\mathscr{M}_{\Delta}$. If
$\frac{n}{2}-id(u_i,u_j)\in \mathscr{M}_{\Delta}$, then
$\frac{n}{2}-id(u_i,u_j)\le 2^{\Delta-1}-1$ and so
$2^{\Delta}-2-id(u_i,u_j)< 2^{\Delta-1}-1$. Now , we have
$2^{\Delta-1}-1<id(u_i,u_j)$ and so
$\frac{n}{2}-id(u_i,u_j)<id(u_i,u_j)$, a contradiction by
definition of index-distance. Therefore, $id(u_i,u_j)\in
\mathscr{M}_{\Delta}$.

Conversely, Assume that $id(u_i,u_j)\in
\mathscr{M}_{\Delta}$. Thus, $id(u_i,u_j)\le 2^{\Delta-1}-1$ and
so $\frac{n}{2}-id(u_i,u_j)\ge\frac{n}{2}-
2^{\Delta-1}+1>2^{\Delta}-2- 2^{\Delta-1}+1= 2^{\Delta-1}-1$.
Therefore, $\frac{n}{2}-id(u_i,u_j)\notin \mathscr{M}_{\Delta}$
and by Corollary \ref{cor1} we have $|N(u_i)\cap N(u_j)|=1$.
\end{proof}

\begin{lemma}\label{lem0000}
Let $W_{\Delta,n}$ be a Kn\"odel graph with vertex set $U\cup V$.
For any non-empty subset $A\subseteq U$:\\
(i) $\underset{v\in N(A)}{\sum}|N(v)\cap A|=\Delta |A|$.\\
(ii) The corresponding cyclic-sequence of $A$ has at most $\Delta |A|-|N(A)|$ elements belonging to $\mathscr{M}_{\Delta}$.
\end{lemma}
\begin{proof}
Let  $A\subseteq U$ be a non-emptyset.\\
(i) It is obvious that the induced subgraph graph $H=W_{\Delta,n}[A\cup N(A)]$ is a bipartite graph and $|E(H)|=\underset{u\in A}{\sum}deg_{H}(u)=\underset{v\in N(A)}{\sum}deg_{H}(v)$, where $E(H)$ is the edge set of $S$. If $u\in A$, then $deg_{H}(u)=\Delta$, and for $v\in N(A)$ we have $deg_{H}(v)=|N(v)\cap A|$. Thus, $\underset{u\in A}{\sum}deg_{H}(u)=\underset{u\in A}{\sum}\Delta=\Delta|A|$ and $\underset{v\in N(A)}{\sum}deg_{H}(v)=\underset{v\in N(A)}{\sum}|N(v)\cap A|$. Consequently,  $\underset{v\in N(A)}{\sum}|N(v)\cap A|=\Delta |A|.$
\\

(ii) Suppose that $A=\{u_{i_1},u_{i_2},\cdots,u_{i_{|A|}}\}$,
where $1\le i_1<i_2<\cdots<i_{|A|}\le \frac{n}{2}$, and let
$n_1,n_2,\cdots,n_{|A|}$ be the corresponding cyclic-sequence of
$A$.  For any vertex $v\in N(A)$, let $r(v)=|N(v)\cap A|$.
Let $J=\{j:n_j\in\mathscr{M}_{\Delta}\}$ and $R=\Delta |A|-|N(A)|$. We prove that $R\ge|J|$. If $R\ge |A|$, then we have nothing to prove, since $|J|\leq |A|$. Assume that
$R< |A|$ and notice that by part (i), $$R=\Delta |A|-|N(A)|=\underset{v\in N(A)}{\sum}|N(v)\cap A|-\sum_{v\in N(A)}1=\underset{v\in
N(A)}{\sum}[r(v)-1].$$ If $\{v\in N(A):r(v)\ge 2\}=\emptyset$, then $R=0$ and $J=\emptyset$, and so  $R\ge|J|$. Thus assume that $\{v\in N(A):r(v)\ge 2\}\neq \emptyset$. Then $R=\underset{\substack{v\in
N(A)\\r(v)\ge2}}{\sum}[r(v)-1].$

Assume that there exists $v'\in N(A)$ such that $r(v')=|A|$. Then $R=r(v')-1+\underset{\substack{v\in N(A)\\r(v)\ge2\\v\ne v'}}{\sum}[r(v)-1]=|A|-1+\underset{\substack{v\in N(A)\\r(v)\ge2\\v\ne v'}}{\sum}[r(v)-1]$. Since $R<|A|$, we obtain that $\underset{\substack{v\in N(A)\\r(v)\ge2\\v\ne v'}}{\sum}[r(v)-1]=0$, $R=|A|-1$, and for each $v\in N(A)\setminus \{v'\}$ we have $r(v)=1$. Since $W_{\Delta,n}$ is vertex transitive, without loss of generality, we assume that $v'=v_{n/2}$.

According to the definition of a Knodel graph, there exist integers $0\le a_{|A|}<a_{|A|-1}<\cdots<a_2<a_1\le\Delta-1$ such that $i_j =\frac{n}{2}-2^{a_j}+1$ for each $1\le j\le|A|$. Moreover, $n_j=i_{j+1}-i_j=2^{a_{i_j}}-2^{a_{i_{j+1}}} \in \mathscr{M}_{\Delta}$ for each $1\le j\le|A|-1$. Evidently, $i_{|A|}-i_{|A|-1}=n_1+n_2+...+n_{|A|-1}=2^{a_{|A|-1}}-2^{a_{|A|}}\in \mathscr{M}_{\Delta}$ and $n_{|A|}=n/2-(i_{|A|}-i_{|A|-1})$. We show that  $n_{|A|}\not \in \mathscr{M}_{\Delta}$. Suppose to the contrary that $n_{|A|} \in \mathscr{M}_{\Delta}$. Since $n_{|A|}=\frac{n}{2}-(i_{|A|}-i_{|A|-1})\in \mathscr{M}_{\Delta}$ and $i_{|A|}-i_{|A|-1}\in \mathscr{M}_{\Delta}$, by Observation \ref{obs1}, $id(u_{i_1},u_{i_{|A|}})\in \mathscr{M}_{\Delta} $ and $\frac{n}{2}-id(u_{i_1},u_{i_{|A|}})\in \mathscr{M}_{\Delta}$, and by Lemma \ref{lem00}, $|N(u_{i_1})\cap N(u_{i_{|A|}})|=2$. Now there exists $v''\ne v_{n/2}$ such that $v''\in N(u_{i_1})\cap N(u_{i_{|A|}})$ and $r(v'')\ge 2$, a contradiction. Therefore, $n_{|A|}\not \in \mathscr{M}_{\Delta}$. Since  $n_j \in \mathscr{M}_{\Delta}$ for each $1\le j\le|A|-1$, we obtain that $|J|=|A|-1=R$. Thus there are at most $R=|A|-1$ elements of the cyclic sequence of $A$ which belong to $\mathscr{M}_{\Delta}$.
 
Next assume that $r(v)<|A|$ for any $v\in N(A)$. Let $X_v=\{j:a_{i_j},a_{i_{j+1}}\in N(A)\cap A\}$. We prove that $J\subseteq \underset{v\in N(A)}\cup X_v$. Let $j\in J$. Then $n_j=i_{j+1}-i_j\in \mathscr{M}_{\Delta}$. By Observation \ref{obs1},  $n_j=i_{j+1}-i_j\in \{id(a_{i_j},a_{i_{j+1}}),\frac{n}{2}-id(a_{i_j},a_{i_{j+1}})\}$ and by Lemma \ref{lem0}, $|N(a_{i_j})\cap N(a_{i_{j+1}})|\ge 1$. Let $v\in N(a_{i_j})\cap N(a_{i_{j+1}})$. Then $a_{i_j},a_{i_{j+1}}\in N(v)\cap A$. Therefore $j\in X_v$ and $j\in \underset{v\in N(A)}\cup X_v$ that implies $J\subseteq \underset{v\in N(A)}\cup X_v$. Then $|J|\le |\underset{v\in N(A)}\cup X_v|$. Observe that $X_v=\{j:a_{i_j}\in N(v)\cap A\}-\{j:a_{i_j}\in N(v)\cap A,a_{i_{j+1}}\notin N(v)\cap A\}$, and
$|\{j:a_{i_j}\in N(v)\cap A\}|=|N(v)\cap A|=r(v)$. Since $N(v)\cap A\not\subseteq A$, we have $\{j:a_{i_j}\in N(v)\cap A,a_{i_{j+1}}\notin N(v)\cap A\}\ne \emptyset$. Therefore $ |X_v|\le r(v)-1$. Consequently, $ |J|\le|\underset{v\in N(A)}\cup X_v|\le \underset{v\in N(A)}\sum |X_v|   \le \underset{v\in N(A)}\sum [r(v)-1]$.
\end{proof}

We remark that one can define the cyclic-sequence and
index-distance for any subset of $V$ in a similar way, and thus
the Observation \ref{obs1}, Lemmas \ref{lem0} and \ref{lem00} and
corollaries  \ref{cor1} and \ref{cor2} are valid for
cyclic-sequence and index-distance on subsets of $V$ as well.

\section{Total domination number of 3-regular Kn\"odel graphs}

We are now ready to determine the total domination number of
$W_{3,n}$. Clearly $n\geq 8$ is an even integer by the
definition of $W_{3,n}$.

\begin{theorem} For each even integer $n\ge8$, $\gamma _t(W_{3,n})=4\lceil \frac{n}{10} \rceil -\left\{\begin{array}{cc}0&n\equiv0,6,8 \emph{ (mod 10)}\\2&n\equiv2,4\emph{ (mod 10)} \end{array}\right.$.

\end{theorem}

\begin{proof}
We divide the proof into five cases depending on $n$.

\textbf{Case 1:}  $n\equiv 0$ (mod $10$). Let $n=10t$, where
$t\ge 1$. Then the set $D_1=\{u_{5k+b},v_{5k+b}: k=0,1,\cdots,
t-1 ; b=1,2  \}$ is a total dominating set for $W_{3,n}$ and thus
$\gamma _t(W_{3,n})\le |D_1|=4t=4\lceil \frac{n}{10} \rceil$. We
show that $\gamma _t(W_{3,n})= 4t$. Suppose to the contrary, that
$\gamma _t(W_{3,n})< 4t$. Let $D$ be a total dominating set with
$4t-1$ elements. Then by the Pigeonhole Principle either $|D\cap U|\le 2t-1$ or $|D\cap V|\le
2t-1$. Without loss of generality, assume that $|D\cap U|\le
2t-1$. Let $|D\cap U|=2t-1-a$, where $a\ge 0$. Then $|D\cap
V|=2t+a$. Observe that $D\cap U$ dominates at most $3|D\cap
U|=6t-3-3a$ vertices of $V$, and so $6t-3-3a\ge 5t=|V|$, since
$D\cap U$ dominates $V$. Clearly the inequality $6t-3-3a\ge 5t$
does not hold if $t\in \{1,2\}$, and thus this contradiction
implies that $\gamma _t(W_{3,n})=4t=4\lceil \frac{n}{10} \rceil$
for $t=1,2$. From here on, assume that $t\ge3$. Thus there are at
most $(6t-3-3a)- 5t=t-3-3a$ vertices of $V$ that are dominated by
at least two vertices of $D\cap U$. In the other words, by Lemma \ref{lem0000} at most
$t-3-3a$ elements of the cyclic-sequence of $D\cap U$ belong to
$\mathscr{M}_3=\{1,2,3\}$. Furthermore, at least
$(2t-1-a)-(t-3-3a)=t+2+2a$ elements of the cyclic-sequence of
$D\cap U$ are greater than 3 (do not belong to $\mathscr{M}_3$ by
Lemma \ref{lem0}). Then by Observation \ref{obs1},
$5t=\underset{i=1}{\overset{2t-1}\Sigma} n_i\ge
4(t+2+2a)+(t-3-3a)=5t+5+5a$, a contradiction. Therefore, $\gamma
_t(W_{3,n})=4t=4\lceil \frac{n}{10} \rceil$.

\textbf{Case 2:}  $n\equiv 2$ (mod $10$). Let $n=10t+2$, where
$t\ge 1$. Then the set $D_2=\{u_{5k+b},v_{5k+b}: k=0,1,\cdots,t-1
; b=1,2  \}\cup \{u_{5t+1},v_{5t+1}\}$ is a total dominating set
for $W_{3,n}$ and thus $\gamma _t(W_{3,n})\le |D_2|=4t+2=4\lceil
\frac{n}{10} \rceil-2$. We show that $\gamma _t(W_{3,n})= 4t+2$.
Suppose to the contrary, that $\gamma _t(W_{3,n})< 4t+2$. Let $D$
be a total dominating set with $4t+1$ elements. Then by the Pigeonhole Principle either
$|D\cap U|\le 2t$ or $|D\cap V|\le 2t$. Without loss of
generality, assume that $|D\cap U|\le 2t$. Let $|D\cap U|=2t-a$,
where $a\ge 0$. Then $|D\cap V|=2t+1+a$. Observe that $D\cap U$
dominates at most $6t-3a$ vertices of $V$ and so  $6t-3a\ge
5t+1=|V|$, since $D\cap U$ dominates $V$. Then there are at most
$(6t-3a)- (5t+1)=t-1-3a$ vertices of $V$ that are  dominated by
at least two vertices of $D\cap U$. By Lemma \ref{lem0000}, at most $t-1-3a$ elements of the cyclic-sequence of $D\cap U$ belong to
$\mathscr{M}_3$. Furthermore, at least
$(2t-a)-(t-1-3a)=t+1+2a$ elements of the cyclic-sequence of
$D\cap U$ are greater than 3 (do not belong to $\mathscr{M}_3$ by
Lemma \ref{lem0}). Then by Observation \ref{obs1},
$5t+1=\underset{i=1}{\overset{2t-a}\Sigma} n_i\ge
4(t+1+2a)+(t-1-3a)=5t+3+5a$, a contradiction. Therefore, $\gamma
_t(W_{3,n})=4t+2=4\lceil \frac{n}{10} \rceil-2$.

\textbf{Case 3:}  $n\equiv 4$ (mod $10$). Let $n=10t+4$, where
$t\ge 1$. Then the set $D_3=\{u_{5k+b},v_{5k+b}: k=0,1,\cdots,t-1
; b=1,2  \}\cup \{u_{5t+1},v_{5t+2}\}$ is a total dominating set
for $W_{3,n}$ and thus $\gamma _t(W_{3,n})\le |D_3|=4t+2=4\lceil
\frac{n}{10} \rceil-2$. We show that $\gamma _t(W_{3,n})=4t+2$.
Suppose to the contrary, that $\gamma _t(W_{3,n})< 4t+2$. Let $D$
be a total dominating set with $4t+1$ elements. Then by the Pigeonhole Principle either
$|D\cap U|\le 2t$ or $|D\cap V|\le 2t$. Without loss of
generality, assume that $|D\cap U|\le 2t$. Let $|D\cap U|=2t-a$,
where $a\ge 0$. Then $|D\cap V|=2t+1+a$. Observe that $D\cap U$
dominates at most $6t-3a$ vertices of $V$, and so $6t-3a\ge
5t+2=|V|$, since $D\cap U$ dominates $V$. Clearly the inequality $6t-3a\ge 5t+2$
does not hold if $t=1$, and thus this contradiction
implies that $\gamma _t(W_{3,n})=4t+2=4\lceil \frac{n}{10} \rceil-2$
for $t=1$. From here on, assume that $t\ge2$. Then there are at most
$(6t-3a)- (5t+2)=t-2-3a$ vertices of $V$ that are dominated by at
least two vertices of $D\cap U$. By Lemma \ref{lem0000}, at most $t-2-3a$ elements of the cyclic-sequence of $D\cap U$ belong to
$\mathscr{M}_3$. Furthermore, at least
$(2t-a)-(t-2-3a)=t+2+2a$ elements of the cyclic-sequence of
$D\cap U$ are greater than 3 (do not belong to $\mathscr{M}_3$ by
Lemma \ref{lem0}). Then by Observation \ref{obs1},
$5t+2=\underset{i=1}{\overset{2t-a}\Sigma} n_i\ge
4(t+2+2a)+(t-2-3a)=5t+6+5a$, a contradiction. Therefore, $\gamma
_t(W_{3,n})=4t+2=4\lceil \frac{n}{10} \rceil-2$.

\textbf{Case 4:}  $n\equiv 6$ (mod 10). Let $n=10t+6$, where
$t\ge 1$. Then the set $D_4=\{u_{5k+b},v_{5k+b}: k=0,1,\cdots,t-1
; b=1,2  \}\cup \{u_{5t+1},v_{5t+1},u_{5t+2},v_{5t+3}\}$ is a
total dominating set for $W_{3,n}$ and thus $\gamma _t(W_{3,n})\le
|D_4|=4t+4=4\lceil \frac{n}{10} \rceil$. We show that $\gamma
_t(W_{3,n})=4t+4$. Suppose to the contrary, that  $\gamma
_t(W_{3,n})< 4t+4$. Let $D$ be a total dominating set with $4t+3$
elements. Then by the Pigeonhole Principle either $|D\cap U|\le 2t+1$ or $|D\cap V|\le 2t+1$.
Without loss of generality, assume that $|D\cap U|\le 2t+1$. Let
$|D\cap U|=2t+1-a$, where $a\ge 0$. Then $|D\cap V|=2t+2+a$.
Observe that $D\cap U$ dominates at most $6t+3-3a$ vertices of
$V$, and so $6t+3-3a\ge 5t+3=|V|$, since $D\cap U$ dominates $V$.
Then there are at most $(6t+3-3a)- (5t+3)=t-3a$ vertices of $V$
that are dominated by at least two vertices of $D\cap U$. By Lemma \ref{lem0000}, at most $t-3a$ elements of the cyclic-sequence of $D\cap U$ belong to
$\mathscr{M}_3$. Furthermore, at least $(2t+1-a)-(t-3a)=t+1+2a$ elements of the
cyclic-sequence of $D\cap U$ are greater than 3 (do not belong to
$\mathscr{M}_3$ by Lemma \ref{lem0}). Then by Observation
\ref{obs1},  $5t+3=\underset{i=1}{\overset{2t+1-a}\Sigma} n_i\ge
4(t+2+2a)+(t-3a)=5t+8+5a$, a contradiction. Therefore, $\gamma
_t(W_{3,n})=4t+4=4\lceil \frac{n}{10} \rceil$.

\textbf{Case 5:}  $n\equiv 8$ (mod 10). Let $n=10t+8$, where
$t\ge 0$. Then the set $D_5=\{u_{5k+b},v_{5k+b}: k=0,1,\cdots,t-1
; b=1,2  \}\cup \{u_{5t+1},v_{5t+2},u_{5t+3},v_{5t+4}\}$ is a
total dominating set for $W_{3,n}$ and thus $\gamma _t(W_{3,n})\le
|D_5|=4t+4=4\lceil \frac{n}{10} \rceil$. We show that $\gamma
_t(W_{3,n})=4t+4$. Suppose to the contrary, that $\gamma
_t(W_{3,n})< 4t+4$. Let $D$ be a total dominating set with $4t+3$
elements. Then by the Pigeonhole Principle either $|D\cap U|\le 2t+1$ or $|D\cap V|\le 2t+1$.
Without loss of generality, assume that $|D\cap U|\le 2t+1$. Let
$|D\cap U|=2t+1-a$ and $a\ge 0$. Then $|D\cap V|=2t+2+a$. Observe
that $D\cap U$ dominates at most $6t+3-3a$ vertices of $V$, and
so $6t+3-3a\ge 5t+4=|V|$, since $D\cap U$ dominates $V$. Then
there are at most $(6t+3-3a)- (5t+4)=t-1-3a$ vertices of $V$ that
are dominated by at least two vertices of $D\cap U$. By Lemma \ref{lem0000}, at most $t-1-3a$ elements of the cyclic-sequence of $D\cap U$ belong to
$\mathscr{M}_3$. Furthermore,
at least  $(2t+1-a)-(t-1-3a)=t+2+2a$ elements of the
cyclic-sequence of $D\cap U$ are greater than 3 (do not belong to
$\mathscr{M}_3$ by Lemma \ref{lem0}). Then by Observation
\ref{obs1},  $5t+4=\underset{i=1}{\overset{2t+1-a}\Sigma} n_i\ge
4(t+2+2a)+(t-1-3a)=5t+7+5a$, a contradiction. Therefore $\gamma
_t(W_{3,n})=4t+2=4\lceil \frac{n}{10} \rceil$.
\end{proof}

\section{Conclusion}
Domination number and total domination number of the $3$-regular
Kn\"odel graphs have already been determined. Determining other
variations of domination (such as connected domination number,
independent domination number, etc.) on these graphs seems of
sufficient interest. Moreover, determining the domination variants
of the $k$-regular Kn\"odel graphs for $k\geq 4$ are still open.

\end{document}